\documentclass{amsart}

\usepackage{amsmath,bm}
\usepackage{amssymb}
\usepackage{amsthm}
\usepackage{mathtools}

\usepackage[abbrev]{amsrefs}
 \usepackage{url}

\newtheorem{theorem}{Theorem}[section]
\newtheorem{lemma}[theorem]{Lemma}

\theoremstyle{definition}

\newtheorem{example}[theorem]{Example}

\theoremstyle{remark}
\newtheorem{remark}[theorem]{Remark}

\theoremstyle{cor}
\newtheorem{cor}[theorem]{Corollary}

\theoremstyle{prop}
\newtheorem{prop}[theorem]{Proposition}

\numberwithin{equation}{section}

\newcommand{\N}{\mathbb{N}}

\newcommand{\R}{\mathbb{R}}
\newcommand{\eps}{\varepsilon}
\newcommand{\zt}{\zeta}
\newcommand{\gm}{\gamma}
\newcommand{\bone}{\mathbf{1}}
\newcommand{\wg}{\wedge}

\DeclareMathOperator*{\esssup}{ess\,sup}
\newcommand{\relmiddle}[1]{\mathrel{}\middle#1\mathrel{}}

\begin{document}

\title{Compactness of semigroups of explosive symmetric Markov processes}


\author{Kouhei Matsuura}
\address{Mathematical Institute, Tohoku University, Aoba, Sendai 980-8578, Japan}
\curraddr{}
\email{kouhei.matsuura.r3@dc.tohoku.ac.jp}
\thanks{}



\date{}

\dedicatory{}

\begin{abstract}
In this paper, we investigate spectral properties of explosive symmetric Markov processes. Under a condition on its $1$-resolvent, we prove the $L^1$-semigroups of Markov processes become compact operators.  
\end{abstract}

\maketitle

\section{Introduction}
Let $E$ be a locally compact separable metric space and $\mu$ a positive Radon measure on $E$ with topological full support. Let $X=(\{X_t\}_{t \ge 0}, \{P_x\}_{x \in E}, \zt)$ be a $\mu$-symmetric Hunt process on $E$. Here $\zt$ is the life time of $X$. We assume $X$ satisfies the irreducible property, resolvent strong Feller property, in addition, {\it tightness}  property, namely, for any $\eps>0$, there exists a compact subset $K \subset E$ such that $\sup_{x \in E}R_{1}\bone_{E \setminus K}(x)<\eps$. Here $R_1$ is the $1$-resolvent of $X$. The family of symmetric Markov processes with these three properties is called {\it Class} (T).

In  \cite{T3}, the spectral properties of a Markov process in Class (T) are studied. For example, if $\mu$-symmetric Hunt process $X$ belongs to Class (T), the semigroup becomes a compact operator on $L^{2}(E,\mu)$. This implies the corresponding non-positive self-adjoint operator has only discrete spectrum. Furthermore, it is shown that the eigenfunctions have bounded continuous versions. The self-adjoint operator is extended to linear operators $(\mathcal{L}^p,D(\mathcal{L}^p))$ on $L^{p}(E,\mu)$ for any $1\le p\le \infty$. In \cite{T2}, it is shown that the spectral bounds of the operators  $(\mathcal{L}^p,D(\mathcal{L}^p))$ are independent of $p \in [1,\infty]$. Then, a question arises: {\it if a $\mu$-symmetric Hunt process $X$ belongs to Class (T), the spectra of $(\mathcal{L}^p,D(\mathcal{L}^p))$ are independent of $p \in [1,\infty]$?}

In this paper, we answer this question by showing that the semigroup of $X$ becomes a compact operator on $L^{1}(E,\mu)$ under some additional conditions. These include the condition that $\lim_{x \to \partial}R_{1}\bone_{E}(x)=0$ which are more restrictive than Class (T). However, it will be proved that for the symmetric $\alpha$-stable process $X^D$ on an open subset $D \subset \R^d$ the following assertions are equivalent (Theorem~\ref{th:4}):
\begin{itemize}
\item[(i)] for any $1\le p \le \infty$, the semigroup of $X^D$ is a compact operator on $L^{p}(D,m)$;
\item[(ii)] the semigroup of $X^D$ is a compact operator on $L^{2}(D,m)$;
\item[(iii)] $\lim_{ |x| \to \infty}E_{x}[\tau_D]=0$;
\item[(iv)] $\lim_{ |x| \to \infty} \int_{0}^{\infty}e^{-t}P_{x}[\tau_D>t]\,dt=0$.
\end{itemize}
Here, $m$ is the Lebesgue measure on $D$ and $\tau_D=\inf\{t>0 \mid X_t^D \notin D\}$. The above conditions are equivalent to 
\begin{itemize}
\item[(iii)'] $\lim_{x \in D,\ |x| \to \infty}E_{x}[\tau_D]=0$
\end{itemize}
provided $D$ is unbounded. In fact, the assertion (iv) is equivalent to the tightness property of $X$. Thus, for the symmetric $\alpha$-stable process $X^D$ on an open subset $D$, the tightness property is equivalent to all assertions in the Theorem~\ref{th:4} mentioned above and implies that the spectra are independent of $p \in [1,\infty]$. The key idea is to give an approximate estimate by the semigroup of part processes by employing Dynkin's formula (Proposition~\ref{prop:on}). 

In \cite[Theorem~4.2]{TTT}, the authors consider the rotationally symmetric $\alpha$-stable process on $\R^d$ with a killing potential $V$. Under a suitable condition on $V$, they proved the tightness propety of the killed stable process. In Example~\ref{ex:kill} below, we will prove the semigroup of the process becomes a compact operator on $L^{1}(\R^d,m)$ under the  assumption on $V$ essentially equivalent to \cite[Theorem~4.2]{TTT}.

In Example~\ref{ex:tc} below, we will consider the time-changed process of the rotationally symmetric $\alpha$-stable process on $\R^d$ by the additive functional $A_t=\int_{0}^{t}W(X_s)^{-1}\,ds$. Here $\alpha \in (0,2]$ and $W$ is a nonnegative Borel measurable function on $\R^d$. The Revuz measure of $A$ is $W^{-1}m$ and the time-changed process $X^W$ becomes a $W^{-1}m$-symmetric Hunt process on $\R^d$. The life time of $X^W$ equals to $A_{\infty}$. To investigate the spectral property of $X^W$ is just to investigate the spectral properties of the operator of the form $\mathcal{L}^W=-W(x)(-\Delta)^{\alpha/2}$ on $L^{2}(\R^d,W^{-1}m)$. When $W(x)=1+|x|^{\beta}$ and $\alpha=2$, it is shown in \cite[Proposition~2.2]{MS} that the spectrum of $\mathcal{L}^W$ is discrete in $L^{2}(\R^d,W^{-1}m)$ if and only if $\beta>2$. When $\alpha \in (0,2)$, $d>\alpha$, and $W(x)=1+|x|^\beta$ with $\beta \ge 0$, it is shown in \cite[Proposition~3.3]{TTT} that 
the supectrum of $\mathcal{L}^W$ in $L^{2}(\R^d,W^{-1}m)$ is discrete if and only if $\beta>\alpha$. This is equivalent to that the semigroup of $X^W$ is a compact operator on $L^{2}(\R^d,W^{-1}m)$ if and only if $\beta>\alpha$. In Theorem~\ref{th:tc} below, we shall prove that if $\beta>\alpha$, the semigroup becomes a compact operator on $L^{1}(\R^d,W^{-1}m)$.

\section{Main results}
Let $E$ be a locally compact separable metric space and $\mu$ a positive Radon measure on $E$. Let $E_{\partial }$ be the its one-point compactification $E_{\partial }=E \cup \{\partial \}$. A $[-\infty,\infty]$-valued function $u$ on $E$ is extended to a function on $E_{\partial }$ by setting $u(\partial )=0$.

Let $X=(\{X_t\}_{t \ge 0}, \{P_x\}_{x \in E}, \zt)$ be a $\mu$-symmetric Hunt process on $E$. The semigroup $\{p_t\}_{t>0}$ and the resolvent $\{R_{\alpha}\}_{\alpha \ge 0}$ are defined as follows:
\begin{align*}
&p_{t}f(x)=E_{x}[f(X_t)]=E_{x}[f(X_t):t<\zt], \\
&R_{\alpha}f(x)=E_{x}\left[\int_{0}^{\zt}\exp(-\alpha t)f(X_t)\,dt \right], \quad f \in \mathcal{B}_{b}(E),\ x \in E.
\end{align*}
Here, $\mathcal{B}_{b}(E)$ is the space of bounded Borel mesurable functions on $E$. $E_x$ denotes the expectation with respect to $P_x$. By the symmetry and the Markov property of $\{p_t\}_{t>0}$, $\{p_t\}_{t>0}$ and $\{R_{\alpha}\}_{\alpha>0}$ are canonically extended to operators on $L^{p}(E,\mu)$ for any $1\le p \le \infty$. The extensions are also denoted by $\{p_t\}_{t>0}$ and $\{R_{\alpha}\}_{\alpha>0}$, respectively.

For an open subset $U \subset E$, we define $\tau_U$ by $\tau_U=\inf \{t>0 \mid X_t \notin U\}$ with the convention that $\inf \emptyset=\infty$. We denote by $X^U$ the part of $X$ on $U$. Namely, $X^U$ is defined as follows. 
\begin{equation*}
X_t^{U}=\begin{cases}
X_t, \quad &t<\tau_U \\
\partial,\quad &t \ge \tau_U.
\end{cases}
\end{equation*}
$X^{U}=(\{X_t^{U}\}_{t\ge0}, \{P_{x}\}_{x \in U})$ also becomes a Hunt process on $U$ with life time $\tau_U$. The semigroup $\{p_t^{U}\}_{t>0}$ is identified with 
\begin{align*}
&p_{t}^{U}f(x)=E_{x}[f(X_t^{U})]=E_{x}[f(X_t):t<\tau_U]
\end{align*}
$\{p_t^U\}_{t>0}$ is also symmetric with respect to the measure $\mu$ restricted to $U$. $\{p_t^U\}_{t>0}$ and $\{R_\alpha^U\}_{\alpha>0}$ are also extended to  operators on $L^{p}(U,\mu)$ for any $1\le p \le \infty$ and the extensions are also denoted by $\{p_t^U\}_{t>0}$ and $\{R_{\alpha}^U\}_{\alpha>0}$, respectively.

We now make the following conditions on the symmetric Markov process $X$.

\begin{enumerate}
\item[{\bf I.}]({\bf Semigroup strong Feller}) For any $t>0$, $p_t(\mathcal{B}_{b}(E)) \subset C_{b}(E)$, where $C_{b}(E)$ is the space of bounded continuous functions on $E$.
\item[{\bf II.}]({\bf Tightness property}) $\lim_{x \to \partial }R_{1}\bone_{E}(x)=0$.
\item[{\bf III.}]({\bf Local $L^\infty$-compactness}) For any $t>0$ and open subset $U \subset E$ with $\mu(U)<\infty$, $p_{t}^{U}$ is a compact operator on $L^{\infty}(U,\mu)$.
\end{enumerate}

\begin{remark}\label{re}
\begin{itemize}
\item[(i)]  By the condition~I, the semigroup kernel of $X$ is absolutely continuous with respcet to $\mu$:
$$p_{t}(x,dy)=p_{t}(x,y)\,d\mu(y).$$
Furthermore, the resolvent of $X$ is strong Feller: for any $\alpha>0$, $R_{\alpha}(\mathcal{B}_{b}(E)) \subset C_{b}(E)$.
\item[(ii)] The conditions~I and II lead us to the tightness property in the sense of \cite{T4,T3}: for any $\eps>0$, there exists a compact subset $K \subset E$ such that $\sup_{x \in E}R_{1}\bone_{E \setminus K}(x)<\eps$. See \cite[Remark~2.1~(ii)]{T4} for details. We denote by $C_{\infty}(E)$ the space of continuous functions on $E$ vanishing at infinity. Under the condition~I and the invariance $R_{1}(C_{\infty}(E)) \subset C_{\infty}(E)$ of $X$, the condition~II is equivalent to the tightness property in the sense of \cite{T4,T3}. See \cite[Remark~2.1~(iii)]{T4} for details. In addition to the conditions~I and II, we assume $X$ is irreducible in the sense of \cite{T4}. Then, by using \cite[Lemma~2.2~(ii), Lemma~2.6, Corollary~3.8]{T4}, we can show $\sup_{x \in E}E_{x}[\exp(\lambda \zt)]<\infty$ for some $\lambda>0$ and thus $R_{0}\bone_{E}$ is bounded on $E$. We further see from the strong Feller property and the resolvent equation of $\{R_{\alpha}\}_{\alpha>0}$ that $R_{0}\bone_{E} \in C_{\infty}(E)$.
\item[(iii)]  The conditions~I and II imply $p_{t}(C_{\infty}(E)) \subset C_{\infty}(E)$ for any $t>0$, and thus $X$ is doubly Feller in the sense of \cite{CK}. This implies that for any $t>0$ and open $U \subset E$, $p_{t}^{U}$ is strong Feller: $p_{t}^{U}(\mathcal{B}_{b}(U)) \subset C_{b}(U)$. See \cite[Theorem~1.4]{CK} for the proof.
\item[(iv)] Let $U \subset E$ be an open subset with $\mu(U)<\infty$. The condition~III is satisfied if the semigroup of $X^U$ is ultracontractive: for any $t>0$ and $f \in L^{1}(U,\mu)$, $p_{t}^{U}f$ belongs to $L^{\infty}(U,\mu)$. Indeed, we see from \cite[Theorem~1.6.4]{Da} that $p_{t}^{U}$ is a compact operator on $L^{1}(U,\mu)$ and so is on $L^{\infty}(U,\mu)$. In particular, if the semigroup of $X$ is ultracontractive, the condition~III is satisfied.
\end{itemize}
\end{remark}

We are ready to state the main result of this paper.

\begin{theorem}\label{th:1}
Assume $X$ satisfies the conditions from I to III. Then, for any $t>0$, $p_t$ becomes a compact operator on $L^{\infty}(E,\mu)$. 
\end{theorem}

By the symmetry of $X$, each $p_{t}:L^{\infty}(E,\mu) \to L^{\infty}(E,\mu)$ is regarded as the dual-operator of $p_{t}:L^{1}(E,\mu) \to L^{1}(E,\mu)$. By using Schauder's theorem, we obtain the next corollary.

\begin{cor}\label{th:2}
Assume $X$ satisfies the conditions from I to III. Then, for any $t>0$, $p_{t}$ becomes a compact operator on $L^{1}(E,\mu)$. 
\end{cor}

Let $(\mathcal{L}^p,D(\mathcal{L}^p))$ be the generator of $\{p_t\}_{t>0}$ on $L^{p}(E,\mu)$, $1\le p\le \infty$.
By using \cite[Theorem~1.6.4]{Da}, we can show the next theorem.
\begin{theorem}\label{th:3}Assume $X$ satisfies the conditions from I to III. Then,
\begin{itemize}
\item[(i)] for any $1\le p \le \infty$ and $t>0$, $p_t$ is a compact operator on $L^{p}(E,\mu)$;
\item[(ii)] spectra of $(\mathcal{L}^p,D(\mathcal{L}^p))$ are independent of $p \in [1,\infty]$ and the eigenfunctions of $(\mathcal{L}^2,D(\mathcal{L}^2))$ belong to $L^{p}(E,\mu)$ for any $1 \le p \le \infty$.
\end{itemize}
\end{theorem}

\section{Proof of Theorem~\ref{th:1}}
Since $E$ is a locally compact separable metric space, there exist increasing bounded open subsets $\{U_n\}_{n=1}^{\infty} $ and compact subsets $\{K_n\}_{n=1}^{\infty}$ such that for any $n \in \N$, $K_n \subset U_n \subset K_{n+1}$ and $E=\bigcup_{n=1}^{\infty}U_n=\bigcup_{n=1}^{\infty}K_n$. We write $\tau_n$ for $\tau_{U_n}$. The semigroup of the part process of $X$ on $U_n$ is simply denoted by $\{p_{t}^n\}_{t>0}$. 

The quasi-left continuity of $X$ yields the next lemma.
\begin{lemma}\label{lem:ql}
For any $x \in E$, $P_{x}(\lim_{n \to \infty}\tau_{n}=\zt)=1$.
\end{lemma}

The following formula is called Dynkin's formula.
\begin{lemma}\label{lem:d}
It holds that
\begin{equation*}
p_{t}f(x)=p_{t}^{U}f(x)+E_{x}[p_{t-\tau_U}f(X_{\tau_U}): \tau_U \le t]
\end{equation*}
for any $x \in E$, $f \in \mathcal{B}_{b}(E)$, $t>0$, and any open subset $U$ of $E$.
\end{lemma}
\begin{proof}
It is easy to see that
\begin{equation}
p_{t}f(x)=p_{t}^{U}f(x)+E_{x}[f(X_t):\tau_U \le t ].\label{eq:eq1}
\end{equation}
Let $n \in \N$. On $\{\tau_U \le t\}$, we define $s_n$ by 
\begin{equation*}
s_n|_{\{(k-1)/2^n \le t-\tau_U <k/2^n\}}=k/2^n,\quad k \in \N.
\end{equation*}
We note that $\lim_{n \to \infty}s_n=t-\tau_U$. By the strong Markov property of $X$, 
\begin{align}
E_{x}[f(X_{\tau_U+s_n}):\tau_U \le t]&=\sum_{k=1}^{\infty}E_{x}[f(X_{\tau_U+k/2^n}):(k-1)/2^n \le t-\tau_U <k/2^n] \notag \\
&=\sum_{k=1}^{\infty}E_{x}[E_{X_{\tau_U}}[f(X_{k/2^n})]:(k-1)/2^n \le t-\tau_U <k/2^n] \notag \\
&=E_{x}[p_{s_n}f(X_{\tau_U}): \tau_U \le t]\label{eq:eqd}.
\end{align}
Letting $n \to \infty$ in \eqref{eq:eqd}, we obtain 
\begin{equation}
E_{x}[f(X_{t}):\tau_U \le t]=E_{x}[p_{t-\tau_U}f(X_{\tau_U}): \tau_U \le t] \label{eq:eq2}
\end{equation}
Combining \eqref{eq:eq1} with \eqref{eq:eq2}, we complete the proof.
\end{proof}
By using Dynkin's formula and the semigroup strong Feller property, we obtain the next lemma.
\begin{lemma}\label{lem:uni}
Let $K$ be a compact subset of $E$. Then, for any $t>0$ and a nonegative $f \in \mathcal{B}_{b}(E)$,
$$
\lim_{n \to \infty}\sup_{x \in K} E_{x}[p_{t-\tau_n}f(X_{\tau_n}): \tau_n \le t]=0.
$$
\end{lemma}
\begin{proof}
We may assume $K \subset U_1$. By the condition~I and Remark~\ref{re}~(iii), both $p_tf$ and $p_{t}^nf$ are continuous on $K$. Hence, we see from Dynkin's formula (Lemma~\ref{lem:d}) that
\begin{equation}
E_{x}[p_{t-\tau_n}f(X_{\tau_n}): \tau_n \le t]=p_{t}f(x)-p_{t}^{n}f(x) \label{eq:1}
\end{equation}
is continuous on $K$. For any $t>0$ and $x \in E$, $p_{t}^{n}f(x) \le p_{t}^{n+1}f(x)$. Hence, (LHS) of \eqref{eq:1} is non-increasing in $n$. By Lemma~\ref{lem:ql}, (LHS) of \eqref{eq:1} converges to 
\begin{align*}
&\lim_{n \to \infty}E_{x}[p_{t-\tau_n}f(X_{\tau_n}): \tau_n \le t]=\lim_{n \to \infty}(p_{t}f(x)-p_{t}^{n}f(x)) \\
&=\lim_{n \to \infty}E_{x}[f(X_t):t \ge \tau_n] =E_{x}[f(X_t):t \ge \zt] \\
&=E_{x}[f(\partial ):t \ge \zt]=0,
\end{align*}
and the proof is complete by Dini's theorem.
\end{proof}

For each $n \in \N$ and $t>0$, we define the operator $T_{n,t}$ on $L^{\infty}(E,\mu)$ by
$$L^{\infty}(E,\mu) \ni f \mapsto  E_{(\cdot)}[p_{t-\tau_n}f(X_{\tau_n}): \tau_n\le t].$$ 
The operator norm of $T_{n,t}$ is estimated as follows.

\begin{prop}\label{prop:on}
Let $n ,m \in \N$ with $m<n$. Then, for any $t>0$,
\begin{align*}
&\| T_{n,t} \|_{L^{\infty}(E,\mu) \to L^{\infty}(E,\mu)} \\
&\le \sup_{x \in K_m}E_{x}[p_{t-\tau_n}\bone_{E}(X_{\tau_n}):\tau_n \le t] + (4/t) \times \sup_{x \in E \setminus K_m}E_{x}[\zt].
\end{align*}
Here, $\|\cdot\|_{L^{\infty}(E,\mu) \to L^{\infty}(E,\mu)}$ denotes the operator norm from $L^{\infty}(E,\mu)$ to itself.
\end{prop}

\begin{proof}
Let $f \in L^{\infty}(E,\mu)$ with $\|f\|_{L^{\infty}(E,\mu)}=1$. Then, we have
\begin{align*}
&\|E_{(\cdot)}[p_{t-\tau_n}f(X_{\tau_n}): \tau_n \le t]\|_{L^{\infty}(E,\mu)} \\
&\le \|f\|_{L^{\infty}(E,\mu)} \times  \esssup_{x \in E}E_{x}[p_{t-\tau_n}\bone_{E}(X_{\tau_n}):\tau_n \le t]\\ 
&\le \esssup_{x \in K_m}E_{x}[p_{t-\tau_n}\bone_{E}(X_{\tau_n}):\tau_n\le t]+\esssup_{x \in E \setminus K_m}E_{x}[p_{t-\tau_n}\bone_{E}(X_{\tau_n}):  t/2 < \tau_n \le t] \\
&\quad +\esssup_{x \in E \setminus K_m}E_{x}[p_{t-\tau_n}\bone_{E}(X_{\tau_n}): \tau_n \le t/2] \\
&\le \sup_{x \in K_m}E_{x}[p_{t-\tau_n}\bone_{E}(X_{\tau_n}):\tau_n \le t]+\sup_{x \in E \setminus K_m}P_{x}(t/2<\tau_n) \\
&\quad+ \sup_{x \in E \setminus K_m}\sup_{s \in [t/2,t]}p_{s}\bone_{E}(x).
\end{align*}
Here, $\esssup$ denotes the essential supremum with respect to $\mu$. 
Moreover, we see $P_{x}(t/2<\tau_n) \le P_{x}(t/2<\zt)\le (2/t) \times E_{x}[\zt]$ and $$p_{s}\bone_{E}(x)=P_{x}(X_s \in E)=P_{x}(s<\zt)\le (1/s)\times E_{x}[\zt].$$ 
Combining these estimates, we obtain
 the following estimate
\begin{align*}
&\| E_{(\cdot)}[p_{t-\tau_n}f(X_{\tau_n}): \tau_n \le t] \|_{L^{\infty}(E,\mu)}\\
&\le \sup_{x \in K_m}E_{x}[p_{t-\tau_n}\bone_{E}(X_{\tau_n}): \tau_n \le t] + (4/t) \times \sup_{x \in E \setminus K_m}E_{x}[\zt].
\end{align*}
\end{proof}
Let $X^{(1)}$ be the $1$-subprocess of $X$. Namely, $X^{(1)}=(\{X_{t}^{(1)}\}_{t \ge 0}, \{P_{x}^{(1)}\}_{x \in E}, \zt^{(1)})$ is the $\mu$-symmetric Hunt process on $E$ whose semigroup $\{p_{t}^{(1)}\}_{t \ge 0}$ is given by 
\begin{equation*}
p_{t}^{(1)}f(x):=E_{x}^{(1)}[f(X_{t}^{(1)})]=E_{x}[e^{-t}f(X_t)],\quad t>0,\ x \in E,\ f \in \mathcal{B}_{b}(E),
\end{equation*}
where $E_{x}^{(1)}$ is the expectation with respect to $P_{x}^{(1)}$. For each $n \in \N$, we denote by $X^{(1),n}$ the part process of $X^{(1)}$ on $U_n$. The semigroup is denoted by $\{p_{t}^{(1),n}\}_{t \ge 0}$. It is easy to see 
\begin{equation}
p_{t}^{(1)}f(x)-p_{t}^{(1),n}f(x)=e^{-t}(p_{t}f(x)-p_{t}^{n}f(x)) \label{eq:eqcommute}
\end{equation}
 for any $t>0$, $x \in E$, $f \in \mathcal{B}_{b}(E)$, and $n \in \N$. 
 
 For each $n \in \N$ and $t>0$, we define the operator $T_{n,t}^{(1)}$ on $L^{\infty}(E,\mu)$ by 
 $$L^{\infty}(E,\mu) \ni f \mapsto  E_{(\cdot)}^{(1)}[p_{t-\tau_{n}'}^{(1)}f(X_{\tau_{n}'}^{(1)}): \tau_{n}'\le t],$$
 where we define $\tau_{n}'=\inf \{t >0 \mid X_{t}^{(1)} \notin U_n \}$. By  using \eqref{eq:eqcommute} and applying Lemma~\ref{lem:d} to $X$ and $X^{(1)}$, we have 
\begin{align}
&T_{n,t}^{(1)}f(x)=p_{t}^{(1)}f(x)-p_{t}^{(1),n}f(x)=e^{-t}(p_{t}f(x)-p_{t}^nf(x)) \notag \\
&=e^{-t} \times E_{x}[p_{t-\tau_n}f(X_{\tau_n}): \tau_n \le t]=e^{-t} \times T_{n,t}f(x) \label{eq:eqcommute2}
\end{align}
for any $t>0$, $n \in \N$, $x \in E$ and $f \in \mathcal{B}_{b}(E)$. By using \eqref{eq:eqcommute2} and Lemma~\ref{lem:uni}, we obtain the next lemma.
\begin{lemma}\label{lem:commute}
\begin{itemize}
\item[(i)]
It holds that
$$
\lim_{n \to \infty}\sup_{x \in K} T_{n,t}^{(1)}f(x)=0
$$
for any compact subset $K \subset E$, $t>0$ and nonegative $f \in \mathcal{B}_{b}(E)$. 
\item[(ii)] It holds that
\begin{equation*}
\|T_{n,t}\|_{L^{\infty}(E,\mu) \to L^{\infty}(E,\mu)}=e^{t} \times \|T_{n,t}^{(1)}\|_{L^{\infty}(E,\mu) \to L^{\infty}(E,\mu)}
\end{equation*}
for any $t>0$ and $n \in \N$.
\end{itemize}
\end{lemma} 
\begin{proof}[Proof of Theorem~\ref{th:1}]
By the condition~III, each $p_{t}^n$ is regarded as a compact operator on $L^{\infty}(E,\mu)$. Therefore it is sufficient to prove $$\lim_{n \to \infty}\| p_{t}-p_{t}^n \|_{L^{\infty}(E,\mu) \to L^{\infty}(E,\mu)}=0.$$ Lemma~\ref{lem:d} lead us to that for any $n \in \N$ and $t>0$
\begin{align}
&\| p_{t}-p_{t}^n  \|_{L^{\infty}(E,\mu) \to L^{\infty}(E,\mu)} \notag \\
&= \sup_{f \in L^{\infty}(E,\mu),\ \|f\|_{L^{\infty}(E,\mu)}=1}  \| E_{(\cdot)}[p_{t-\tau_n}f(X_{\tau_n}): \tau_n \le t] \|_{L^{\infty}(E,\mu)} \notag \\
&=\|T_{n,t}\|_{L^{\infty}(E,\mu) \to L^{\infty}(E,\mu)}. \label{eq:commute3}
\end{align}
It holds that $E_{x}^{(1)}[\zt^{(1)}]=R_{1}\bone_{E}(x)$ for any $x \in E$. Applying Proposition~\ref{prop:on} to $X^{(1)}$, we have
\begin{align}
&\|T_{n,t}^{(1)}\|_{L^{\infty}(E,\mu) \to L^{\infty}(E,\mu)}  \notag \\
&\le \sup_{x \in K_m}E_{x}^{(1)}[p_{t-\tau_{n}'}^{(1)}\bone_{E}(X_{\tau_{n}'}^{(1)}):\tau_{n}' \le t] + (4/t) \times \sup_{x \in E \setminus K_m}E_{x}^{(1)}[\zt^{(1)}] \notag \\
&=  \sup_{x \in K_m}T_{n,t}^{(1)}\bone_{E}(x) + (4/t) \times \sup_{x \in E \setminus K_m}R_{1}\bone_{E}(x). \label{eq:commute4}
\end{align}
Combining \eqref{eq:commute3}, \eqref{eq:commute4} and Lemma~\ref{lem:commute}~(ii), we have
\begin{align*}
&\| p_{t}-p_{t}^n  \|_{L^{\infty}(E,\mu) \to L^{\infty}(E,\mu)} \le e^{t} \times \left\{   \sup_{x \in K_m}T_{n,t}^{(1)}\bone_{E}(x) + (4/t) \times \sup_{x \in E \setminus K_m}R_{1}\bone_{E}(x) \right\}.
\end{align*}
Letting $n \to \infty$ and then $m \to \infty$, the proof is complete by Lemma~\ref{lem:commute}~(i).
\end{proof}

\section{Examples}
\begin{example}
Let $\alpha \in (0,2]$ and $X$ be the rotationally symmetic $\alpha$-stable process on $\R^d$. If $\alpha=2$, $X$ is identified with the $d$-dimensional Brownian motion. Let $D \subset \R^d$ be an open subset of $\R^d$ and $X^{D}$ be the $\alpha$-stable process on $D$ with Dirichlet boundary condition. Since $X$ is semigroup doubly Feller in the sense of \cite{CK}, the condition I is satisfied for $X^D$. Since the semigroup of $X$ is ultracontractive, so is the semigroup of $X^D$. Thus, the condition III is also satisfied. It is shown in \cite[Lemma~1]{KM} that the semigroup of $X^D$ is a compact operator on $L^{2}(D,m)$ if and only if $\lim_{|x| \to \infty}E_{x}[\tau_D]=0$.

Hence, by using Theorem~\ref{th:2} and Theorem~\ref{th:3}, we obtain the next theoem.
\end{example}
\begin{theorem}\label{th:4} The following are equivalent:
\begin{itemize}
\item[(i)] for any $1\le p \le \infty$, the semigroup of $X^D$ is a compact operator on $L^{p}(D,m)$;
\item[(ii)] the semigroup of $X^D$ is a compact operator on $L^{2}(D,m)$;
\item[(iii)] $\lim_{ |x| \to \infty}E_{x}[\tau_D]=0$;
\item[(iv)] $\lim_{ |x| \to \infty} \int_{0}^{\infty}e^{-t}P_{x}[\tau_D>t]\,dt=0$.
\end{itemize}
\end{theorem}
\begin{remark}\label{rem:HC}
The semigroup of $X^D$ is not necessarily a Hilbert-Schmidt operator but can be a compact operator on $L^{1}(D,m)$. Namely, there exists an open subset $D \subset \R^d$ which satisfies the following conditions:
\begin{itemize}
\item[(D.1)]
$\lim_{ |x| \to \infty}E_{x}[\tau_D]=0$;
\item[(D.2)] the trace of the semigroup of $X^D$ is infinite. 
\end{itemize}
For example, let $\alpha=2$, $d \in \N$, and $$D=\bigcup_{n=1}^{\infty} D_n:=\bigcup_{n=1}^\infty B(e_n,r_n)$$
Here, $B(e_n,r_n) \subset \R^d$ denotes the open ball centered at $e_n=(n,0,\cdots,0) \in \R^d$ with radius $r_n=\{\log \log(n+3)\}^{-1/2}$. It is easy to see $r_n > 1$ for $n>24$. We shall check $D$ satisfies the conditions (D.1) and (D.2). We denote by $p_{t}^{D_n}(x,y)$ the heat kernel density of $X^{D_n}$ with respect to $m$. By \cite[Theorem~1.9.3]{Da},
\begin{align*}
&\int_{D}p_{t}^{D}(x,x)\,dm(x)\ge \sum_{n=25}^{\infty}\int_{D_n}p_{t}^{D_n}(x,x)\,dm(x) \\
&\ge \sum_{n=25}^{\infty} (8 \pi t)^{-d/2} \times r_n \times \exp(-8 \pi^2 dt/r_n^2)\\
&\ge (8 \pi t)^{-d/2}\sum_{n=25}^{\infty} \{ \log(n+3) \}^{-1/2-8 \pi^2 dt}=\infty.
\end{align*}
Therefore, the trace of the semigroup of $X^D$ is infinite. On the other hand, for any $x \in D_n$, 
\begin{align*}
&E_{x}[\tau_D]=E_{x}[\tau_{D_n}] \le E_{o}[\tau_{B(|e_n-x|+r_n)}].
\end{align*}
Here, $o$ denotes the origin of $\R^d$ and $B(|e_n-x|+r_n)$ denotes the open ball centered at the origin with radius $|e_n-x|+r_n$. $|e_n-x|$ is the length of $e_n-x$. Since $|e_n-x| \le r_n$, it holds that
$$E_{o}[\tau_{B(|e_n-x|+r_n)}]=(|e_n-x|+r_n)^2/d \le 4r_n^2/d.$$
Since $r_n \to 0$ as $n \to \infty$, $\lim_{ |x| \to \infty}E_{x}[\tau_D]=0$.
\end{remark}
\begin{example}\label{ex:kill}
Let $\alpha \in (0,2]$ and $X=(\{X_t\}_{t \ge 0}, \{P_x\}_{x \in \R^d},\zt)$ be the rotationally symmetric $\alpha$-stable process on $\R^d$. The semigroup of $X$ is denoted by $\{p_t\}_{t>0}$. Let $V$ be a positive Borel measurable function on $\R^d$ with the following properties:
\begin{itemize}
\item[(V.1)] $V$ is locally bounded. Namely, for any relatively compact open subset $G \subset \R^d$, $\sup_{x \in G}V<\infty$;
\item[(V.2)] $\lim_{x \in \R^d,\ |x| \to \infty}V(x)=\infty$.
\end{itemize}

We set $A_t=\int_{0}^{t}V(X_s)\,ds$. Let $X^V=(\{X_t\}_{t \ge 0}, \{P_{x}^{V}\}_{x \in \R^d},\zt)$ be the subprocess of $X$ defined by $dP_{x}^{V}=\exp(-A_t)dP_x$. The semigroup $\{p_{t}^{V}\}_{t>0}$ is identified with 
\begin{align*}
&p_{t}^{V}f(x)=E_{x}[\exp(-A_t)f(X_t)],
\quad f \in \mathcal{B}_{b}(\R^d),\ x \in \R^d.
\end{align*}

\begin{theorem}\label{thm:kill}
$X^V$ satisfies the conditions from I to III.
\end{theorem}

Before proving Theorem~\ref{thm:kill}, we give a lemma. We denote by $B(n)$ the open ball of $\R^d$ centered at the origin $o$ and radius $n \in \N$. The semigroup of $X$ is doubly Feller in the sense of \cite{CK}. Thus, for any $n \in \N$, the semigroup of $X^{B(n)}$ is strong Feller.
\begin{lemma}\label{lem:locuni} It holds that
\begin{equation*}
\lim_{n \to \infty}\sup_{x \in K}P_{x}(\tau_{B(n)} \le t)=0
\end{equation*}
for any $t>0$ and compact subset $K \subset \R^d$. Here, $\tau_{B(n)}=\inf\{t>0 \mid X_t \in \R^d \setminus B(n)\}$.
\end{lemma}
\begin{proof}
Without loss of generality, we may assume $K \subset B(1)$. For any $t>0$, $n \in \N$, and $x \in \R^d$, 
\begin{align*}
P_{x}(\tau_{B(n)} \le t)&=\bone_{\R^d}(x)-P_{x}(\tau_{B(n)}>t)\\
&=\bone_{\R^d}(x)-p_{t}^{B(n)}\bone_{\R^d}(x).
\end{align*}
Thus, we see from the strong Feller property of $X^{B(n)}$ that for any $n \in \N$, $P_{\cdot}(\tau_{B(n)} \le t)$ is a continuous function on $K$. It follows from the conservativesness of $X$ and Lemma~\ref{lem:ql} that for any $x \in \R^d$,
$$\varlimsup_{n \to \infty}P_{x}(\tau_{B(n)} \le t)\le P_{x}(\zt \le t)=0$$
 and the convergence is non-increasing. The proof is complete by Dini's theorem.
\end{proof}

\begin{proof}[Proof of Theorem~\ref{thm:kill}]
Since the semigroup of $X$ is ultracontractive, so is the semigroup of $X^V$. Hence, the condition~III is satisfied. We will check $X^V$ satisfies the condition~I. Let $K$ be a compact subset of $\R^d$ and take $n_0 \in \N$ such that $K \subset B(n_0)$. Then, for any $s \in (0,1)$ and $n>n_0$,
\begin{align*}
 &\sup_{x \in K}E_{x}[1-\exp(-A_s )]\\
 & \le \sup_{x \in K}E_{x}[A_{s \wg \tau_{B(n)}}]+\sup_{x \in K}P_{x}(\tau_{B(n)} \le s) \\
 &=\sup_{x \in K}E_{x}\left[\int_{0}^{ s\wg \tau_{B(n)}}V(X_t)\,dt\right]+\sup_{x \in K}P_{x}(\tau_{B(n)} \le 1)=:I_1+I_2.
 \end{align*}
 By the condition~(V.1), $\lim_{s \to 0}I_1=0$. By Lemma~\ref{lem:locuni}, $\lim_{n \to \infty}I_2=0$. Thus, 
 \begin{equation}
 \lim_{s \to 0}\sup_{x \in K}E_{x}[1-\exp(-A_s )]=0.\label{eq:eqa}
 \end{equation}
Let $t>0$ and $f \in \mathcal{B}_{b}(\R^d)$. Since the semigroup of $X$ is strong Feller, for any $s \in (0,t)$, $p_{s}p_{t-s}^{V}f$ is continuous on $\R^d$. By using \eqref{eq:eqa}, we have 
\begin{align*}
&\varlimsup_{s \to 0}\sup_{x \in K}\left| p_{t}^{V}f(x)-p_{s}p_{t-s}^{V}f(x)  \right|\\
&=\varlimsup_{s \to 0}\sup_{x \in K} \left| E_{x}[\exp(-A_t)f(X_{t})]-E_{x}[p_{t-s}^{V}f(X_s)]\right| \\
&=\varlimsup_{s \to 0}\sup_{x \in K}\left|  E_{x} [\exp(-A_s)E_{X_{s}}[\exp(-A_{t-s})f(X_{t-s})]]-E_{x}[p_{t-s}^{V}f(X_s)] \right|\\
&\le \|f\|_{L^\infty(\R^d,m)} \times \varlimsup_{s \to 0}\sup_{x \in K}E_{x}[1-\exp(-A_s )]=0.
\end{align*}
This means that the semigroup of $X^V$ is strong Feller and the condition~I is satisfied.
 
 Finally, we shall show the condition~II. Let $x \in \R^d$ and $t>0$. Since $X$ is spatially homogeneous, 
\begin{equation*}
P_{x}^{V}(\zt>t)=E_{x}\left[\exp \left(-\int_{0}^{t}V(X_s)\,ds\right)\right]=E_{o}\left[\exp \left(-\int_{0}^{t}V(x+X_s)\,ds\right)\right].
\end{equation*}
It follows from the condition~(V.2) that for any $t>0$, $\lim_{x \in \R^d,\ |x| \to \infty}P_{x}^{V}(\zt>t)=0$. By the positivity of $V$, we can show that $\sup_{x \in \R^d}P_{x}^V(\zt>t)<1$ for any $t>0$. By the additivity of $\{A_t\}_{t \ge 0}$, 
\begin{align*}
P_{x}^{V}(\zt>t+s)&=E_{x}[\exp(-A_{t+s}):t+s<\zt]\\
&=E_{x}[\exp(-A_s)E_{X_s}[\exp(-A_t):t<\zt]:s<\zt] \\
&\le \sup_{x \in \R^d}P_{x}^{V}(\zt>t) \times  \sup_{x \in \R^d}P_{x}^{V}(\zt>s)
\end{align*}
for any $x \in \R^d$ and $s,t>0$. Hence, letting $p= \sup_{x \in \R^d}P_{x}^{V}(\zt>1)<1 $, we have
\begin{align*}
\sup_{x \in \R^d}E_{x}^{V}[\zt]&=\sup_{x \in \R^d}\int_{0}^{\infty}P_{x}^{V}(\zt>t)\,dt   \le \sum_{n=0}^{\infty}\int_{n}^{n+1}\sup_{x \in \R^d}P_{x}^{V}(\zt>n)\,dt \\
&\le 1+\sum_{n=1}^{\infty}p^{n}=1/(1-p).
\end{align*}
We denote by $p_{t}^{V}(x,y)$ the heat kernel density of $X^V$. For any $\eps>0$, 
\begin{align}
E_{x}^{V}[\zt]&\le \eps+E_{x}^{V}[E^{V}_{X_{\eps}^V}[\zt]] \le \eps+\int_{\R^d}p_{\eps}^{V}(x,y)E^{V}_{y}[\zt]\,dm(y) \notag \\
&\le \eps+\frac{1}{1-p}\times P_{x}^{V}(\zt>\eps). \notag
\end{align}
By letting $x \to \infty$, we have $\varlimsup_{x \in \R^d,\ |x| \to \infty}E_{x}^{V}[\zt] \le \eps$. Since $\eps$ is chosen arbitrarily, the condition~II is satisfied.
\end{proof}
\end{example}

\begin{example}\label{ex:tc}
Let $\alpha \in (0,2]$ and $d >\alpha $, and $X=(\{X_t\}_{t \ge 0}, \{P_x\}_{x \in \R^d},\zt)$ be the rotationally symmetric $\alpha$-stable process on $\R^d$. We note that $X$ is transient. Let us consider the additive functional $\{A_t\}_{t \ge 0}$ of $X$ defined by 
\begin{equation*}
A_t=\int_{0}^{t}W(X_s)^{-1}\,ds,\quad t \ge 0.
\end{equation*}
Here $W$ is a Borel measurable function on $\R^d$ with the condition:
\begin{equation*}
1+|x|^{\beta} \le W(x) <\infty ,\quad x \in \R^d,
\end{equation*}
where $\beta \ge 0$ is a constant. The Revuz measure of $\{A_t\}_{t \ge 0}$ is identified with $W^{-1}\,m$. Denote $\mu=W^{-1}m$. $\mu$ is not necessary a finite measure on $\R^d$. Noting that $A_t$ is continuous and strictly increasing in $t$, we define $X^{\mu}=(\{X_t^{\mu}\}_{t \ge 0}, \{P_x\}_{x \in \R^d},\zt^{\mu})$ by 
\begin{equation*}
X_t^{\mu}=X_{\tau_t},\ t \ge0,\quad \tau=A^{-1},\quad \zt^{\mu}=A_{\infty}.
\end{equation*}
Then, $X^{\mu}$ becomes a $\mu$-symmetric Hunt process on $\R^d$. $X^\mu$ is transient because the transience is preserved by time-changed transform (\cite[Theorem~6.2.3]{FOT}). The semigroup and the resolvent of $X^{\mu}$ are denoted by $\{p_t^{\mu}\}_{t>0}$, $\{R_{\alpha}^{\mu}\}_{\alpha \ge 0}$, respectively. 
\begin{theorem}\label{th:tc}
If $\beta>\alpha$, $X^{\mu}$ satisfies the conditions from I to III.
\end{theorem}

Before proving Theorem~\ref{th:tc}, we give some notions and lemmas. Let $(\mathcal{E},\mathcal{F})$ be the Dirichlet form of $X$. $(\mathcal{E},\mathcal{F})$ is identified with
\begin{align*}
\mathcal{E}(f,g)&=\frac{K(d,\alpha)}{2}\int_{\R^d}\hat{f}(x)\hat{g}(x)\,|x|^{\alpha}dx,\\
f,g \in \mathcal{F}&=\left\{f \in L^{2}(\R^d,m) \relmiddle|  \int_{\R^d}|\hat{f}(x)|^2\,|x|^{\alpha}dx<\infty \right\}.
\end{align*}
Here $\hat{f}$ denotes the Fourier transform of $f$ and $K(d,\alpha)$ is a positive constant. Recall that $m$ is the Lebesgue measure on $\R^d$. $m$ is also denoted by $dx$. Let $(\mathcal{E},\mathcal{F}_e)$ denotes the extended Dirichlet space of $(\mathcal{E},\mathcal{F})$, Namely, $\mathcal{F}_e$ is the family of Lebesgue measurable functions $f$ on $\R^d$ such that $|f|<\infty$ $m$-a.e. and there exists a sequence $\{f_n\}_{n=1}^{\infty}$ of functions in $\mathcal{F}$ such that $\lim_{n \to \infty}f_n=f$ $m$-a.e. and $\lim_{n,k\to \infty}\mathcal{E}(f_n-f_k,f_n-f_k)=0$. $\{f_n\}_{n=1}^{\infty}$ as above called an {\it approximating sequence} for $f \in \mathcal{F}_e$ and  $\mathcal{E}(f,f)$ is defined by 
$\mathcal{E}(f,f)=\lim_{n \to \infty}\mathcal{E}(f_n,f_n).$
Since the quasi support of $\mu$ is identified with $\R^d$, the Dirichlet form $(\mathcal{E}^\mu,\mathcal{F}^\mu)$ of $X^\mu$ is described as follows (see \cite[Theorem~6.2.1, (6.2.22)]{FOT} for details).
\begin{align*}
\mathcal{E}^\mu(f,g)&=\mathcal{E}(f,g), \quad
\mathcal{F}^{\mu}= \mathcal{F}_e\cap L^{2}(\R^d,\mu).
\end{align*}
By identifying the Dirichlet form of $X^\mu$, we see that the semigroup of $X^{\mu}$ is ultracontractive.
\begin{lemma}\label{lem:uc}
For any $f \in L^{1}(\R^d,\mu)$ and $t>0$, $p_t^{\mu}f \in L^{\infty}(\R^d,\mu)$. 
\end{lemma}
\begin{proof}
By \cite[Theorem~1, p138]{EG} for $\alpha=2$ and \cite[Theorem~6.5]{DPV} for $\alpha \in (0,2)$, there exist positive constants $C>0$ and $q \in (2,\infty)$ such that
\begin{equation}
\left\{ \int_{\R^d}|f|^{q}\,d\mu \right\}^{2/q} \le \left\{ \int_{\R^d}|f|^{q}\,dm \right\}^{2/q} \le C\mathcal{E}(f,f),\quad f \in \mathcal{F}.\label{eq:sobolev}
\end{equation}
Let $\{f_n\}_{n=1}^{\infty} \subset \mathcal{F}$ be  an approximating sequence of $f \in \mathcal{F}^{\mu}=\mathcal{F}_e \cap L^{2}(\R^d,\mu)$. By using Fatou's lemma and \eqref{eq:sobolev}, we have
\begin{align*}
\left\{ \int_{\R^d}|f|^{q}\,d\mu \right\}^{2/q} \le \varliminf_{n \to \infty} \left\{ \int_{\R^d}|f_n|^{q}\,d\mu \right\}^{2/q}\le C\varliminf_{n \to \infty} \mathcal{E}(f_n,f_n)=C\mathcal{E}(f,f).
\end{align*}
The proof is complete by \cite{CKS}. See also \cite[Theorem~4.2.7]{FOT}.
\end{proof}

Let $U$ be an open subset of $\R^d$ and $X^{\mu,U}$ be the part of $X^\mu$ on $U$:
\begin{equation*}
X_t^{\mu,U}=\begin{cases}
X_t^{\mu}, \quad &t<T_U:=\inf\{t>0 \mid X_t^{\mu} \notin U\} \\
\partial,\quad &t \ge T_U.
\end{cases}
\end{equation*}
 The semigroup and the resolvent are denoted by $\{p_{t}^{\mu,U}\}_{t>0}$ and $\{R_{\gamma}^{\mu, U}\}_{\gamma>0}$, respectively. 

\begin{lemma}\label{lem:str}
Let $f \in \mathcal{B}_{b}(U)$, $\gm>0$, and $U \subset \R^d$ be a open subset. Then, $R_{\gm}^{\mu, U}f \in C_{b}(\R^d)$. In particular, for each $\gm>0$ and $x \in U$, the kernel $R_{\gm}^{\mu,U}(x,\cdot)$ is absolutely continuous with respect to $\mu|_{U}$.
\end{lemma}
\begin{proof}
It is easy to see that $\lim_{t \to 0}\sup_{x \in \R^d}E_{x}[A_t]=0$. This means that $\mu$ is in the Kato class of $X$ in the sense of \cite{KKT}. Since the resolvent of $X$ is doubly Feller in the sense of \cite{KKT}, by \cite[Theorem~7.1]{KKT}, the resolvent of $X^{\mu}$ is also doubly Feller. By using \cite[Theorem~3.1]{KKT}, we complete the proof. ``In particular'' part follows from the same argument as in \cite[Exercise~4.2.1]{FOT}.
\end{proof}

Following the arguments in \cite[Theorem~5.1]{AK}, we strengthen Lemma~\ref{lem:str} as follows. 
\begin{prop}\label{prop:str}
Let $f \in \mathcal{B}_{b}(U)$, $t>0$, and $U \subset \R^d$ be a bounded open subset. Then, $p_{t}^{\mu,U}f \in C_{b}(U)$.
\end{prop}
\begin{proof}
Step~1: We denote by $(\mathcal{L}_{U}, D(\mathcal{L}_{U}))$ the non-positive generator of $\{p_{t}^{\mu,U}\}$ on $L^{2}(U,\mu)$. By Lemma~\ref{lem:uc}, $-\mathcal{L}_{U}$ has only discrete spectrum. Let $\{\lambda_n\}_{n=1}^{\infty} \subset [0,\infty)$ be the eigenvalues of $-\mathcal{L}_{U}$ written in increasing order repeated according to multiplicity, and let $\{\varphi_n\}_{n=1}^{\infty} \subset  D(\mathcal{L}_{U})$ be the corresponding eigenfunctions: $-\mathcal{L}_{U}\varphi_n=\lambda_n \varphi_n$. Then, $\varphi_n=e^{\lambda_n} p_{1}^{\mu,U}\varphi_n \in L^{\infty}(\R^d,\mu)$ by Lemma~\ref{lem:uc}. Hence,  for each $n \in \N$, there exists a bounded measurable version of $\varphi_n$ (still denoted as $\varphi_n$). By Lemma~\ref{lem:str}, for each $\gm>0$ and $n \in \N$, $R_{\gm}^{\mu,U}\varphi_n$ is continuous on $U$. Furthermore, we see from  \cite[Theorem~4.2.3]{FOT} that
\begin{equation}
R_{\gm}^{\mu,U}\varphi_n=(\gm-\mathcal{L}_U)^{-1}\varphi_n=(\gm+\lambda_n)^{-1}\varphi_n\quad \mu\text{-a.e. on }U \label{eq:identity}.
\end{equation}
Therefore, there exists a (unique) bounded continuous version of $\varphi_n$ (still denoted as $\varphi_n$). By \cite[Theorem~2.1.4]{Da}, the series 
\begin{equation}
p_{t}^{\mu,U}(x,y):=\sum_{n=1}^{\infty}e^{-\lambda_n t}\varphi_n(x) \varphi_n(y)\label{eq:expan}
\end{equation}
absoluetely converges uniformly on $[\eps,\infty) \times U \times U$ for any $\eps>0$. Since $\{\varphi_n\}_{n=1}^{\infty}$ are bounded continuous on $U$, $p_{t}^{\mu,U}(x,y)$ is also continuous on $(0,\infty) \times U \times U$ and defines an integral kernel of $\{ p_{t}^{\mu, U}\}_{t>0}$. Namely, for each $t>0$ and $f \in L^{2}(U,\mu)$,
\begin{equation}
p_{t}^{\mu,U}f(x)=\int_{U}p_{t}^{\mu,U}(x,y)f(y)\, d\mu(y)\quad\text{for $\mu$-a.e. }x\in U. \label{eq:expan2}
\end{equation}
The uniform convergence of the series \eqref{eq:expan} imply the boundedness of $p_{t}^{\mu,U}(x,y)$ on $[\eps,\infty) \times U \times U$ for each $\eps>0$. We also note that $p_{t}^{\mu,U}(x,y) \ge 0$ by \eqref{eq:expan2} and the fact that $p_{t}^{\mu,U}f \ge 0$ $\mu$-a.e. for any $f \in L^{2}(U,\mu)$ with $f \ge 0$.

Step~2: In this step, we show that for each $x \in U$, $\gm>0$, and $f \in \mathcal{B}_{b}(\R^d)$,
\begin{equation}
\int_{0}^{\infty}e^{-\gm t}E_{x}[f(X_{t}^{\mu,U})]\,dt=\int_{0}^{\infty}e^{-\gm t}\left(\int_{U}p_{t}^{\mu, U}(x,y)f(y)\,d\mu(y)\right)\,dt\label{eq:laplace}.
\end{equation}
By the absolute continuity of $R_{\gm}^{\mu,U}$ (Lemma~\ref{lem:str}), for any $\eps>0$, 
\begin{align*}
&\int_{\eps}^{\infty}e^{-\gm t}E_{x}[f(X_{t}^{\mu,U})]\,dt=e^{-\gm \eps}R_{\gm}^{\mu,U}(p_{\eps}^{\mu,U}f)(x) \\
&=e^{-\gm \eps}R_{\gm}^{\mu,U}\left(\sum_{n=1}^{\infty}e^{-\lambda_n \eps}\left(\int_{U}\varphi_n(y)f(y)\,d\mu(y) \right)\varphi_n \right)(x)\\
&=\sum_{n=1}^{\infty}e^{-(\gm+\lambda_n)\eps}(\gm+\lambda_n)^{-1}\left(\int_{U}\varphi_n(y)f(y)\,d\mu(y) \right)\varphi_n(x).
\end{align*}
Here, we used the identity \eqref{eq:identity} and the uniform convergence of the series \eqref{eq:expan}. Set $$a_n^{\eps}=e^{-(\gamma+\lambda_n)\eps}(\gamma+\lambda_n)^{-1}=\int_{\eps}^{\infty}e^{-(\gamma+\lambda_n)t}\,dt.$$  Since the series \eqref{eq:expan} uniformly converges on $[\eps,\infty) \times U \times U$ for each $\eps>0$, 
\begin{align}
&\int_{\eps}^{\infty}e^{-\gm t}E_{x}[f(X_{t}^{\mu,U})]\,dt=\sum_{n=1}^{\infty}a_{n}^{\eps}\left(\int_{U}\varphi_n(y)f(y)\,d\mu(y) \right)\varphi_n(x) \notag \\
&=\sum_{n=1}^{\infty}\int_{\eps}^{\infty}\int_{U}e^{-\lambda_nt}\varphi_n(y)\varphi_n(x) f(y) \,d\mu(y)e^{-\gm t}\,dt \notag \\
&=\int_{\eps}^{\infty}\left(\int_{U}p_{t}^{\mu,U}(x,y)f(y)\,d\mu(y) \right)e^{-\gamma t}\,dt. \label{eq:laplace2}
\end{align}
By letting $\eps \to 0$ in \eqref{eq:laplace2}, we obtain \eqref{eq:laplace}.

Step~3: By \eqref{eq:laplace} and the uniquness of Laplace transforms, it holds that
\begin{equation}
E_{x}[f(X_{t}^{\mu,U})]=\int_{U}p_{t}^{\mu, U}(x,y)f(y)\,d\mu(y)\quad dt\text{-a.e. }t\in(0,\infty) \label{eq:laplace3}
\end{equation}
for any $x \in E$ and $f \in \mathcal{B}_{b}(\R^d)$. If $f$  is bounded continuous on $U$, by the continuity of $X_t^{\mu}$ and $p_{t}^{\mu,U}(x,y)$, \eqref{eq:laplace3} holds for any $t \in (0,\infty)$. By using a monotone class argument, we have
\begin{equation*}
E_{x}[f(X_{t}^{\mu,U})]=\int_{U}p_{t}^{\mu,U}(x,y)f(y)\,d\mu(y)
\end{equation*}
for any $x \in E$ and $f \in \mathcal{B}_{b}(\R^d)$, and $t>0$. By Step~1, for each $t>0$, $p_t^{\mu,U}(x,y)$ is bounded continuous on $U \times U$. Since $\mu(U)<\infty$, the proof is complete by dominated convergence theorem.
\end{proof}

\begin{cor}\label{cor:str}
For any $f \in \mathcal{B}_{b}(\R^d)$ and $t>0$, $p_{t}^{\mu}f \in C_{b}(\R^d)$.
\end{cor}
\begin{proof}
Let $K$ be a compact subset of $\R^d$. For any bounded open subset $U \subset \R^d$ with $K \subset U$,
\begin{align*}
&\sup_{x \in K}|p_{t}^{\mu}f(x)-p_{t}^{\mu,U}f(x)|\le \|f\|_{L^{\infty}(E,\mu)}\times \sup_{x \in K}P_{x}[t \ge T_U].
 \end{align*}
By Proposition~\ref{prop:str}, $p_{t}^{\mu,U}f$ is continuous on $K$. By Lemma~\ref{lem:ql} and Dini's theorem, $$\lim_{U \nearrow \R^d} \sup_{x \in K}P_{x}[t \ge T_U]=0,$$ which complete the proof.
\end{proof}

\begin{proof}[Proof of Theorem~\ref{th:tc}]
By Lemma~\ref{lem:uc} and Corollary~\ref{cor:str}, the conditions~I and III are satisfied. We shall prove the condition~II.  Let $\gamma_1,\gamma_2>0$ such that $\gamma_1<d$ and $\gamma_1+\gamma_2>d$. Setting 
\begin{equation*}
J_{\gamma_1,\gamma_2}(x)=\int_{\R^d}\frac{dy}{|x-y|^{\gamma_1}(1+|y|^{\gamma_2})}\quad x \in \R^d,
\end{equation*}
$J_{\gamma_1,\gamma_2}$ is bounded on $\R^d$ and there exist positive constants $c_1,c_2,c_3$ such that 
\begin{equation}
J_{\gamma_1,\gamma_2}(x)\le \begin{cases}
c_1|x|^{d-(\gamma_1+\gamma_2)}, & \text{if } \gamma_2<d, \\
c_2(1+|x|)^{-\gamma_1} \log |x| & \text{if }\gamma_2=d, \\
c_3(1+|x|)^{-\gamma_1} &\text{if } \gamma_2>d
\end{cases} \label{eq:eqbounds}
\end{equation}
for any $x \in \R^d$. See \cite[Lemma~6.1]{MS}  for the bounds \eqref{eq:eqbounds}. 

 We denote by $G(x,y)$ the Green function of $X$. It is known that
\begin{equation*}
G(x,y)=c(d,\alpha) |x-y|^{\alpha-d}.
\end{equation*}
Here $c(d,\alpha)=2^{1-\alpha}\pi^{-d/2}\Gamma((d-\alpha)/2)\Gamma(\alpha/2)^{-1}$ and $\Gamma$ is the gamma function: $$\Gamma(s)=\int_{0}^{\infty}x^{s-1}\exp(-x)\,dx.$$ Recall that $\beta>\alpha$. Since 
\begin{align*}
R_{0}^{\mu}\bone_{\R^d}(x)&=\int_{\R^d}G(x,y)\,d\mu(y)\le c(d,\alpha) \int_{\R^d}\frac{dy}{|x-y|^{d-\alpha}W(y)} \\
& \le c(d,\alpha) \int_{\R^d}\frac{dy}{|x-y|^{d-\alpha}(1+|y|^{\beta})}\\
&=c(d,\alpha)J_{d-\alpha,\beta}(x),
\end{align*}
$R_{0}^{\mu}\bone_{\R^d}$ is bounded on $\R^d$ and  $\lim_{x \in \R^d, |x| \to \infty}R_{0}^{\mu}\bone_{\R^d}(x)=0$.
\end{proof}
\end{example}

\noindent
{\bf Acknowledgements.} The author would like to thank referees for their valuable comments and suggestions which improve the quality of the paper.
He also would like to thank Professor Masayoshi Takeda for helpful comments and encouragement.
He also would like to thank Professors Kwa\'snicki Mateusz, Masanori Hino and Naotaka Kajino for their helpful comments on Remark~\ref{rem:HC}.

\end{document}